\documentclass[11pt,a4paper,openany]{amsart}

\usepackage{amsmath,amsfonts,ams}
\usepackage{latexsym}
\usepackage{amssymb}
\usepackage{color}
\usepackage{epsfig}
\usepackage{graphicx}
\usepackage{subfigure}
\usepackage{mathrsfs}
\usepackage{amscd}
\usepackage{amsthm}
\usepackage{color}
\usepackage{cite}
\usepackage{epstopdf}

\numberwithin{equation}{section}
\newtheorem{proposition}{Proposition}[section]
\newtheorem{definition}{Definition}[section]
\newtheorem{lemma}{Lemma}[section]
\newtheorem{theorem}{Theorem}[section]
\newtheorem{corollary}{Corollary}[section]
\newtheorem{remark}{Remark}[section]

\let\pa=\partial

\def\C{\mathop{\mathbb C\kern 0pt}\nolimits}
\def\DD{\mathop{\bf D\kern 0pt}\nolimits}
\def\K{\mathop{\bf K\kern 0pt}\nolimits}
\def\N{\mathop{\mathbb N\kern 0pt}\nolimits}
\def\Q{\mathop{\bf Q\kern 0pt}\nolimits}

\newcommand{\beq}{\begin{equation}}
\newcommand{\eeq}{\end{equation}}
\newcommand{\ben}{\begin{eqnarray}}
\newcommand{\een}{\end{eqnarray}}
\newcommand{\beno}{\begin{eqnarray*}}
\newcommand{\eeno}{\end{eqnarray*}}

\makeatletter
\newcommand{\Extend}[5]{\ext@arrow0099{\arrowfill@#1#2#3}{#4}{#5}}
\makeatother

\begin{document}
\title[Defocusing energy-critical wave with cubic convolution ]{The Defocusing energy-critical  wave Equation with a Cubic Convolution}

\author{Changxing Miao}
\address{Institute of Applied Physics and Computational Mathematics, P. O. Box 8009,\ Beijing,\ China,\ 100088;}
\email{miao\_changxing@iapcm.ac.cn}

\author{Junyong Zhang}
\address{Department of Mathematics, Beijing Institute of Technology, Beijing 100081}
\email{zhang\_junyong@bit.edu.cn}

\author{Jiqiang Zheng}
\address{The Graduate School of China Academy of Engineering Physics, P. O. Box 2101, Beijing, China, 100088}
\email{zhengjiqiang@gmail.com}

\maketitle

\begin{abstract}
In this paper, we study the theory of the global well-posedness and
scattering for the energy-critical wave equation with a cubic
convolution nonlinearity $u_{tt}-\Delta u+(|x|^{-4}\ast|u|^2)u=0$ in
spatial dimension $d \geq 5$.  The main difficulties are the absence
of the classical finite speed of propagation (i.e. the monotonic
local energy estimate on the light cone), which is a fundamental
property to show the global well-posedness and then to obtain
scattering for the wave equations with the local nonlinearity
$u_{tt}-\Delta u+|u|^\frac4{d-2}u=0$. To compensate it, we resort to
the extended causality and utilize the strategy derived from
concentration compactness ideas. Then, the proof of the global
well-posedness and scattering is reduced to show the nonexistence of
the three enemies: finite time blowup; soliton-like solutions and
low-to-high cascade. We will utilize the Morawetz estimate, the
extended causality and the potential energy concentration to
preclude the above three enemies.
\end{abstract}

\begin{center}
 \begin{minipage}{120mm}

 \end{minipage}
 \end{center}

\section{Introduction}
This paper is devoted to study the global well-posedness (GWP) and
scattering for the defocusing energy-critical wave equation with a
cubic convolution (wave-Hartree)
\begin{equation} \label{equ1}
\left\{ \aligned
    &\ddot u - \Delta u    +f(u)=0 , ~~(t,x)\in\R\times\R^d,~~d\geq 5\\
     &(u(0),\dot u(0))=(u_0(x),u_1(x))\in\dot{H}^1(\R^d)\times L^2(\R^d),
\endaligned
\right.
\end{equation}
 where  $u(t,x)$ is a real-valued function in spacetime
$\mathbb{R} \times \mathbb{R}^d$, $f(u)=\big(V(\cdot)*|u|^2\big)u$
with $V(x)=|x|^{-4}$,  the dot denotes the time derivative and $*$
stands for the convolution in $\mathbb{R}^{d}$.

The terminology ``Energy-critical" is due to the fact that both the
energy $E(u,\dot{u})$ defined by
\begin{align}\label{eq1.4}
    E(u,\dot{u}):=\frac12\int_{\R^d}\Big(|\nabla u(x)|^2+|
\dot{u}(x)|^2\Big)dx+\frac14\iint_{\R^{d}\times\R^d}\frac{|u(y)|^2|u(x)|^2}{|x-y|^{4}}dxdy
\end{align}
and the equation \eqref{equ1} itself are invariant under the
rescaling symmetry
\begin{equation}\label{scale}
u(t,x)\mapsto \lambda^{\frac{d-2}{2}}u(\lambda t,\lambda x),
\end{equation}
for $\lambda>0.$ Note that the energy is conserved by the flow
\eqref{equ1}, hence we do not specify time in the notation.

On one hand, the scattering theory for the energy critical wave
equation \begin{equation}\label{energy} \ddot u - \Delta u
+\mu|u|^{2^*-2}u=0,~2^\ast={2d}/{(d-2)},\end{equation} has been
intensively studied in \cite{ BG,BCLPZ, Gri90, Kapi94,KM1,ShaStr94,
 T07}. When $\mu=1$, which corresponds to the
defocusing case, the theory of the global well-posedness and
scattering has been studied by Grillakis \cite{Gri90}, Kapitanski
\cite{Kapi94}, Shatah-Struwe\cite{ShaStr94},
Bahouri-G$\acute{e}$rard \cite{BG}, Tao \cite{T07} and the
references cited therein. In particular, Tao in \cite{T07} derived a
exponential type spacetime bound. The analogs for 3D quintic
Schr\"odinger equation have been established by Colliander, Keel,
Staffilani, Takaoka, and Tao \cite {CKSTT07} . Thereafter,
Killip and Visan \cite{KV2011} gave alternant proof for the  3D
quintic Schr\"odinger equation. We also refer to  Ryckman-Visan
\cite{RV2007}, Visan\cite{Visan2007,Visanphd,Visan2011} for the
defocusing energy-critical Schr\"odinger equation in dimensions
$d\geq4$.

For \eqref{energy} in the focusing case: $\mu=-1$, and $3\leq d\leq
5$, recently Kenig and Merle \cite{KM1} employed the sophisticated
``concentrated compactness " and ``rigidity" method to obtain the
dichotomy-type result under the assumption that $E (u_0, u_1) < E
(W, 0)$, where $W$ denotes the ground state of the elliptic equation
$$\Delta W+|W|^\frac4{d-2}W=0.$$ Thereafter,  \cite{BCLPZ}
extend the above result in \cite{KM1} to higher dimensions. The
analogs for the energy-critical focusing nonlinear Schr\"{o}dinger
equation in the radial case for dimensions $3$ and $4$ have been
established by Kenig and Merle \cite{KM}. While we refer to Killip
and Visan \cite{KV} for the focusing energy-critical Schr\"odinger
equation in dimensions $d\geq5$.

On the other hand, the scattering theory for the Hartree equation
$$i\dot{u}=-\Delta u+\mu(|x|^{-\gamma}*|u|^2)u$$
has been also  studied by many authors (see
\cite{GiV00,LiMZ09,MXZ07a}). For the subcritical cases in the
defocusing case (i.e. $2\leq\gamma<\min\{4,d\}$ and $\mu=1$) ,
Ginibre and Velo \cite{GiV00} derived the associated Morawetz
inequality and extracted an useful Birman-Solomjak type estimate to
obtain the asymptotic completeness in the energy space. Nakanishi
\cite{Na99d} improved the results by establishing a new Morawetz estimate. For
the critical case ($\gamma=4,~d\geq 5$), Miao, Xu and Zhao
\cite{MXZ07a} took advantage of a new kind of the localized Morawetz
estimate to rule out the possibility of the energy concentration at
origin and established the scattering results in the energy space
for the radial data in dimension $d\geq 5$.

For the equation $(\ref{equ1})$ with $V(x)=|x|^{-\gamma}$, using the
ideas of Strauss \cite{St81a,St81b}  and Pecher \cite{Pe85},
Mochizuki \cite{Mo89} showed that if $d \geq 3$, $2\leq \gamma <
\min\{d, 4\}$, then the global well-posedness and scattering results
with small initial data hold in the energy space $H^1(\R^d)\times
L^2(\R^d)$.  One may also refer to \cite{MZ} which study a complete
scattering theory of the Klein-Gordon equation with a cubic
convolution for large data in the subcritical case. This paper is
devoted to study a complete scattering theory of the equation
$(\ref{equ1})$ for the critical case (i.e. $\gamma=4,~d\geq5$) in
the energy space $\dot{H}^1_x(\R^d)\times L^2_x(\R^d)$.

\vskip 0.2in

Our main result is the following global well-posedness and
scattering result in the energy space.

\begin{theorem}\label{thm1.1}
Let $d\geq 5$ and $(u_0, u_1)\in \dot{H}^1(\R^d)\times L^2(\R^d)$ be
initial data with energy bound
\begin{equation}
E(u_0,u_1)\leq E
\end{equation}
for any constant $E>0$. Then there exists a unique global solution
$u(t)$ to \eqref{equ1}. Moreover the solution scatters in the sense that
 there exists  solution $v_\pm$ of the free wave equation
 \begin{equation}\label{le}
    \ddot{v} - \Delta v   =  0
\end{equation}
with $(v_\pm(0), \dot{v}_\pm(0))\in \dot H^1(\R^d)\times L^2(\R^d)$ such that
\begin{equation}\label{1.2}
\big\|\big(u(t), \dot{u}(t)\big)-\big(v_\pm(t),
\dot{v}_\pm(t)\big)\big\|_{\dot H^1_x\times L^2_x} \longrightarrow
0,\quad \text{as}\quad t\longrightarrow \pm\infty.\end{equation}
\end{theorem}

As we know,  there is no pointwise criteria for the critical
problem, GWP and scattering result are simultaneously solved in
general. However, the study history of the $\dot H^1$-critical wave
equation shows us scattering result is later than global
well-posedness.

Now, we recall the history of the $\dot H^1$-critical wave equation
\eqref{energy} with $\mu=1$. In \cite{Gri90,Kapi94,ShaStr94}, by the
finite propagation speed of wave equation, they considered the
Cauchy problem with compact data. And without loss of generality,
one can assume the solution is smooth. They showed the existence of
the global smooth solution by ruling out the accumulation of the
energy at any time, where they utilized the classical finite speed
of propagation (i.e. the monotonic local energy estimate on the
light cone)
\begin{align}\label{}
\int_{|x|\leq R-t}e(t,x)dx\leq \int_{|x|\leq R }e(0,x)dx,~~t>0
\end{align}where\begin{align}\label{}
  e(t,x):=\frac12  | \dot{u}
 |^2  + \frac12  | \nabla u  |^2   + \frac1{2^\ast} |
u|^{2^\ast},
\end{align}
which is a fundamental property for the wave equations with local
nonlinearities. By compactness argument, one can show the global
existence and uniqueness of the energy solution. While the
scattering result of the energy solution was solved ten years later
\cite{BG} by making use of the concentration compactness idea. It
also depends heavily on the monotonic local energy estimate on the
light cone.

 However, the wave-Hartree lacks the classical finite speed of propagation. The nonlocal property of Hartree term
cause the essential difficulties for nonlinear pointwise estimates,
this defeats our attempts to establish the same classical finite
speed of propagation as above. As a substitute,  one may resort to
the causality (Theorem 3 in Menzala-Strauss \cite{MeS82}), however
it holds only for the case $V\in L^{d/3}+L^\infty$,  which does not
contain the energy critical case ($V(x)=|x|^{-4}$), the exponent $d/3$ stems from the
estimate of the term $\int u_t u(V*u^2)dx$ as we know that this term
cannot be controlled by the energy if $V\in L^p_x(\R^d)$ when
$p<{\frac{d}{3}}.$ To overcome it, we make use of the  finite speed
of propagation of the free operators $K(t)$ and $\dot K(t)$ and the
boundness of the local-in-time Strichartz estimate of the solution
(the nonlinear interaction is actually the linear feedback), to
establish the causality for the energy critical case $V\in
L^{\frac{d}{4}-}+L^\infty$. See the detail in Subsection 2.3.

Since the wave-Hartree  lacks the classical finite speed of
propagation, we will not utilize the classical methods in
\cite{Gri90, Kapi94,ShaStr94} to prove the GWP first and then
scattering for the wave equation with local nonlinearity. While,
inspired by the strategy derived from concentration compactness
ideas \cite{KM,KM1,KV11}, we will show GWP and scattering result simultaneously. We
remark that the method here also works for the local nonlinearity
$f(u)=|u|^\frac4{d-2}u$.

\vskip 0.1in Other than the classical finite speed of propagation,
 it is also not easy to verify that the Hartree nonlinearity
satisfies some positive properties, e.g. $G(u)\geq 0$
($G(u)=f(u){\bar u}-2\int_0^{|u|}f(r)dr$), and  it plays an important
role in establishing some
 Morawetz-type estimates
 in \cite{Na99b}. We overcome this
difficulty by using the symmetry property of $V(x)$ and also
establish Morawetz-type estimate  by borrowing some strategies from
\cite{Na99e}.

\vspace{0.2cm}

\subsection{Sketch the proof of Theorem \ref{thm1.1}} Let  $I$ be the maximal lifespan of the solution $u$ to \eqref{equ1}; see \cite{KV11} for the definition of solution.
Define the scattering size \begin{equation}\label{scattersize}
S_I(u)=\|u\|_{L_{t,x}^\frac{2(d+1)}{d-2}(I\times\R^d)}\triangleq
\|u\|_{S(I)}.
\end{equation}
For each $E>0$, let us define
$\Lambda(E)$ to be the quantity
\begin{align*}
\Lambda(E)&=\sup\Big\{\big\|u\big\|_{S(I)};~E(u,u_t)\leq E\Big\}
\end{align*}
where  $u$ ranges over all solutions to \eqref{equ1} on the
spacetime slab $I\times\R^d$ of energy less than $E$ and
$$E_{crit}=\sup\{E: \ \Lambda(E)<+\infty\}.$$

We prove that $E_{crit}=+\infty$ by contradiction argument. If $E_{crit}<+\infty,$ then we will see that
the failure of Theorem \ref{thm1.1} is caused by a special class of
solutions. These solutions must have some good
properties so that they do not exist. Thus we get a contradiction.
For convenience, we recall the definition of almost periodicity modulo
symmetries in \cite{KV}
\begin{definition}\label{apms}
Let $d\geq5,$ and let $u$ to be a solution of \eqref{equ1} with maximal lifespan $I$.  We say that $u$ is
almost periodic modulo symmetries if $(u,u_t)$ is bounded in
$\dot{H}_x^1\times L_x^2$ and there exist functions
$N(t):~I\to\R^+,~x(t):~I\to\R^d$ and $C(\eta):\R^+\to\R^+$ such that
for all $t\in I$ and $\eta>0$,
\begin{equation}\label{apss}
\int_{|x-x(t)|\geq\frac{C(\eta)}{N(t)}}\big(|\nabla
u(t,x)|^2+|u_t(t,x)|^2\big)dx\leq\eta
\end{equation}
and
\begin{equation}\label{apsf}
\int_{|\xi|\geq C(\eta)N(t)}\big(|\xi|^2|\hat u(t,\xi)|^2+|\hat
u_t(t,\xi)|^2\big)d\xi\leq\eta.
\end{equation}
Here $N(t)$ is called the frequency scale function,  $x(t)$ is called the spatial center function, and
$C(\eta)$ is called the compactness modulus function.
\end{definition}
\begin{remark}\label{rem1.1}
As a consequence of Ascoli-Arzela Theorem, $u$ is almost periodic modulo symmetry if
and only if the set
$$\Big\{\big(N(t)^{-\frac{d-2}2}u,~N(t)^{-\frac{d}2}u\big)\Big(t,x(t)+\frac{x}{N(t)}\Big),~t\in
I\Big\}$$ is a compact set in $\dot{H}^1_x\times L_x^2$. If $u$ is almost
periodic modulo symmetries, then
\begin{equation}\label{apsp}
\int_{|x-x(t)|\geq\frac{C(\eta)}{N(t)}}|u(t,x)|^\frac{2d}{d-2}dx\leq\eta.
\end{equation}

\end{remark}
To prove
Theorem \ref{thm1.1}, we recall that
\begin{theorem}[Reduction to almost periodic solutions]\label{reduce} Assume $E_{crit}<+\infty.$ Then there exists a
maximal-lifespan solution $u:~I\times\R^d\to\R$ to \eqref{equ1} such
that
\begin{enumerate}
\item $u$ is almost periodic modulo symmetries;

\item $u$ blows up both forward and backward in time;

\item $u$ has the minimal kinetic energy among all blowup solutions. More precisely, let
$v:~J\times\R^d\to\R$ be a maximal-lifespan solution which blows up
in at least one time direction, then
$$\sup_{t\in I}\big\|\big(u(t),u_t(t)\big)\big\|_{\dot{H}^1_x\times
L^2_x}\leq\sup_{t\in
J}\big\|\big(v(t),v_t(t)\big)\big\|_{\dot{H}^1_x\times L^2_x}.$$
\end{enumerate}
Here $u$ blows up forward in time means that there exists a time
$t_0\in I$ such that $S_{[t_0,\sup I)}(u)=+\infty$; similarly, $u(t,x)$
blows up backward in time  means that $S_{(\text{inf}~ I,t_0]}(u)=+\infty$.

\end{theorem}
For this standard technique, we refer the reader to Keraani \cite{Ker}, Kenig and Merle \cite{KM} and Killip and Visan \cite{KV11}.

So far, we do not have any control on the frequency scale function
$N(t)$. However, the following theorem shows that no matter how
small the set of minimal kinetic energy blowup solution is, we will
inevitably encounter at least one of the following three enemies.
Thus the proof of Theorem \ref{thm1.1} is reduced to showing the
nonexistence of the three scenarios.
\begin{theorem}[Three enemies, \cite{KV11}]\label{three} Let $d\geq5$ and assume
Theorem \ref{thm1.1} fails, that is, $E_{crit}<+\infty.$
Then there exists a maximal-lifespan solution $u:~I\times\R^d\to\R$,
which is almost periodic modulo symmetries, and $S_I(u)=+\infty.$
In addition, the lifespan $I$ and the frequency
scale function $N(t):~I\to\R^+$ satisfy one of the following three
scenarios:
\begin{enumerate}
\item $($Finite time blowup$)$ $|\text{inf}~I|<+\infty,$ or sup~$I<+\infty.$

\item $($Soliton-like solution$)$ $I=\R$ and $N(t)=1$ for all $t\in\R.$

\item $($Low-to-high cascade$)$ $I=\R$ and
$$\inf_{t\in\R}N(t)\geq1,~\text{and}~\varlimsup_{t\to\infty}N(t)=+\infty.$$
\end{enumerate}
\end{theorem}
For more detail, see \cite{KV} for Schr\"odinger equation;  and \cite{KV11} for wave equation. We will
utilize the Morawetz estimate, the extended causality and the
potential energy concentration to preclude the above three enemies.
And one can adopt the proof of Lemma 5.18 in \cite{KVnote} to prove
a similar result for wave equation that the almost periodic
solutions satisfy the following local constancy property:

\begin{lemma}[Local constancy]\label{localconstant} Let $u$ be an almost periodic solution
to \eqref{equ1} on $I$. Then there exists $\delta=\delta(u)$ such
that for all $t_0\in I$,
\begin{equation}
[t_0-\delta N(t_0)^{-1},t_0+\delta N(t_0)^{-1}]\subset I
\end{equation}
and
\begin{equation}
N(t)\sim N(t_0)~ \text{uniformly for}~t\in [t_0-\delta
N(t_0)^{-1},t_0+\delta N(t_0)^{-1}],~t_0\in I.
\end{equation}
\end{lemma}

As a direct consequence of Lemma \ref{localconstant},  we have the
lower bound of $N(t)$ in the spirit of Corollary 5.19 in
\cite{KVnote}, which will play an important role in precluding the
finite time blow up solutions.

\begin{corollary}[the lower bound of $N(t)$]\label{finitetime} Let
$u$ be a non-zero maximal-lifespan solution to \eqref{equ1} with
lifespan $I$ that is almost periodic modulo symmetries with
frequency scale function $N(t):~I\to\R^+$. If $T$ is any finite
endpoint of $I$, then
\begin{equation}\label{lbn}
N(t)\geq \frac{C}{|T-t|},
\end{equation}
in particular, $\lim\limits_{t\to T}N(t)=+\infty.$
\end{corollary}
\begin{proof}
See Corollary 5.19 in \cite{KVnote}.
\end{proof}

The paper is organized as follows. In Section $2$, we deal with the
local theory
 for the equation \eqref{equ1} and the extended causality.  In Section $3$, we show the Morawetz estimate. We prove the potential energy concentration
 for the almost periodic solutions in Section $4$. In Section $5$, we preclude the
 global almost periodic solutions to \eqref{equ1} in the sense of
 Theorem \ref{three}. Finally in Section $6$, we exclude the finite
 time blowup solutions to \eqref{equ1} in the sense of Theorem
 \ref{three}.

\subsection{Notations}
Finally, we conclude the introduction by giving some notations which
will be used throughout this paper. To simplify the expression of
our inequalities, we introduce some symbols $\lesssim, \thicksim,
\ll$. If $X, Y$ are nonnegative quantities, we use $X\lesssim Y $ or
$X=O(Y)$ to denote the estimate $X\leq CY$ for some $C$ which may
depend on the critical energy $E_{crit}$ but not on any parameter
such as $\eta$  and $\rho$, and $X \thicksim Y$ to denote the
estimate $X\lesssim Y\lesssim X$. We use $X\ll Y$ to mean $X \leq c
Y$ for some small constant $c$ which is again allowed to depend on
$E_{crit}$. We will sometimes write $a-$ to denote $a-\eta$ for
arbitrarily small $\eta > 0$. We use $C\gg1$ to denote various large
finite constants, and $0< c \ll 1$ to denote various small
constants. For any $r, 1\leq r \leq \infty$, we denote by $\|\cdot
\|_{r}$ the norm in $L^{r}=L^{r}(\mathbb{R}^d)$ and by $r'$ the
conjugate exponent defined by $\frac{1}{r} + \frac{1}{r'}=1$.

The Fourier transform on $\mathbb{R}^d$ is defined by
\begin{equation*}
\aligned \widehat{f}(\xi):= \big( 2\pi
\big)^{-\frac{d}{2}}\int_{\mathbb{R}^d}e^{- ix\cdot \xi}f(x)dx ,
\endaligned
\end{equation*}
giving rise to the fractional differentiation operators
$|\nabla|^{s}$ and $\langle\nabla\rangle^s$,  defined by
\begin{equation*}
\aligned
\widehat{|\nabla|^sf}(\xi):=|\xi|^s\hat{f}(\xi),~~\widehat{\langle\nabla\rangle^sf}(\xi):=\langle\xi\rangle^s\hat{f}(\xi),
\endaligned
\end{equation*} where $\langle\xi\rangle:=1+|\xi|$.
This helps us to define the homogeneous and inhomogeneous Sobolev
norms
\begin{equation*}
\big\|f\big\|_{\dot{H}^s(\R^d)}:= \big\|
|\xi|^s\hat{f}\big\|_{L^2(\R^d)},~~\big\|f\big\|_{{H}^s(\R^d)}:=
\big\|
\langle\xi\rangle^s\hat{f}\big\|_{L^2(\R^d)},~\|f\|_{\dot{H}^{s,p}(\R^d)}=\big\||\nabla|^sf\big\|_{L^p(\R^d)}.
\end{equation*}

We will also need the Littlewood-Paley projection operators.
Specifically, let $\varphi(\xi)$ be a smooth bump function adapted
to the ball $|\xi|\leq 2$ which equals 1 on the ball $|\xi|\leq 1$.
For each dyadic number $N\in 2^{\mathbb{Z}}$, we define the
Littlewood-Paley operators
\begin{equation*}
\aligned \widehat{P_{\leq N}f}(\xi)& :=
\varphi\Big(\frac{\xi}{N}\Big)\widehat{f}(\xi), \\
\widehat{P_{> N}f}(\xi)& :=
\Big(1-\varphi\Big(\frac{\xi}{N}\Big)\Big)\widehat{f}(\xi), \\
\widehat{P_{N}f}(\xi)& :=
\Big(\varphi\Big(\frac{\xi}{N}\Big)-\varphi\Big(\frac{2\xi}{N}\Big)\Big)\widehat{f}(\xi).
\endaligned
\end{equation*}
Similarly we can define $P_{<N}$, $P_{\geq N}$, and $P_{M<\cdot\leq
N}=P_{\leq N}-P_{\leq M}$, whenever $M$ and $N$ are dyadic numbers.
We will frequently write $f_{\leq N}$ for $P_{\leq N}f$ and
similarly for the other operators.

The Littlewood-Paley operators commute with derivative operators,
the free propagator, and the conjugation operation. They are
self-adjoint and bounded on every $L^p_x$ and $\dot{H}^s_x$ space
for $1\leq p\leq \infty$ and $s\geq 0$, moreover, they also obey the
following
 Bernstein estimates
\begin{eqnarray*}\label{bernstein}
\big\||\nabla|^s P_{\leq N} f \big\|_{L^p} & \lesssim  & N^{s}
\big\|
P_{\leq N} f \big\|_{L^p},    \\
\big\| P_{ N} f \big\|_{L^q} & \lesssim &
N^{\frac{d}{p}-\frac{d}{q}} \big\|P_{ N} f \big\|_{L^p},
\end{eqnarray*}
where  $s\geq 0$ and $1\leq p\leq q \leq \infty$.

\section{Preliminaries}

\subsection{The Strichartz estimates}
In this section, we consider the Cauchy problem  for the equation
$(\ref{equ1})$
\begin{equation} \label{equ2.1}
    \left\{ \aligned &\ddot{u} - \Delta u   +  f(u)=  0, \\
    &u(0)=u_0,~\dot{u}(0)=u_1.
    \endaligned
    \right.
\end{equation}

The integral equation for the Cauchy problem $(\ref{equ2.1})$ can be
written as
\begin{equation}\label{inte1}
u(t)=\dot{K}(t)u_0 + K(t)u_1-\int^{t}_{0}K(t-s)f(u(s))ds,
\end{equation}
or
\begin{equation}\label{inte2}
{u(t)\choose \dot{u}(t)} = V_0(t){u_0(x) \choose u_1(x)}
-\int^{t}_{0}V_0(t-s){0 \choose f(u(s))} ds,
\end{equation}
where
$$ V_0(t) = {\dot{K}(t)\quad K(t)
\choose \ddot{K}(t)\quad \dot{K}(t)},\quad
K(t)=\frac{\sin(t\omega)}{\omega},\quad
\omega=\big(-\Delta\big)^{1/2}.$$

The Strichartz estimates involve the following definitions:
\begin{definition}[Admissible pairs]\label{def1}
A pair of Lebesgue space exponents $(q,r)$ are called wave
admissible for $\mathbb{R}^{1+d}$, or denote by $(q,r)\in
\Lambda_{0}$ when $q,r\geq 2$, and
\begin{equation}\label{equ21}
\frac{2}{q}\leq(d-1)\Big(\frac{1}{2}-\frac{1}{r}\Big),~\text{and}\quad(q,r,d)\neq
(2,\infty,3).
\end{equation}

For a fixed spacetime slab $I\times\R^d$, we define the Strichartz
norm
\begin{equation}\label{suoyimdy}
\|u\|_{S^1(I)}:=\sup\big\||\nabla|^\mu u\big\|_{L_t^q
L_x^r(I\times\R^d)},
\end{equation}
 where the supremum is taken over
all admissible pairs $(q,r)\in\Lambda_0$ and numbers $\mu\in[0,1]$
obeying the scaling condition
$\frac1q+\frac{d}{r}=\frac{d}2-(1-\mu).$ We denote $S^1(I)$ to be
the closure of all test functions under this norm.
\end{definition}

Now we recall the following Strichartz estimates.
\begin{lemma}[Strichartz estimates,\cite{GiV95, KeT98,
LS}]\label{strichartzest} Fix $d\geq5.$ Let $I$ be a compact time
interval and  let $u:~I\times\R^d\to\R$ be a solution to the forced
wave equation
$$u_{tt}-\Delta u+F_1+F_2=0.$$
Then for any $t_0\in I$,
\begin{equation*}
\|u\|_{S^1(I)}+\|\pa_tu\|_{L_t^\infty L_x^2(I\times\R^d)}\lesssim
\big\|\big(u(t_0),\pa_t u(t_0)\big)\big\|_{\dot{H}^1\times
L^2}+\big\||\nabla|^\frac12F_1\big\|_{L_{t,x}^\frac{2(d+1)}{d+3}(I\times\R^d)}+\|F_2\|_{L_t^1
L_x^2(I\times\R^d)}.
\end{equation*}

\end{lemma}

Now we give a few basic estimates.

\begin{lemma} [Product rule \cite{CW}]\label{moser}
Let $s\geq0$, and $1<r,p_j,q_j<\infty$ such that
$\frac1r=\frac1{p_i}+\frac1{q_i}~(i=1,2).$ Then,
$$\big\||\nabla|^s(fg)\big\|_{L_x^r(\R^d)}\lesssim\|f\|_{{L_x^{p_1}(\R^d)}}\big\||\nabla|^sg
\big\|_{{L_x^{q_1}(\R^d)}}+\big\||\nabla|^sf\big\|_{{L_x^{p_2}(\R^d)}}\|g\|_{{L_x^{q_2}(\R^d)}}.$$
\end{lemma}

This together with Hardy-Littlewood-Sobolev inequaltiy yields the
following nonlinear estimate.

\begin{lemma}[Nonlinear estimate]\label{nonle}  For $d\geq 5$, we
have
\begin{equation}
\begin{split}
&\Big\||\nabla|^\frac12\Big[\big(|x|^{-4}\ast(uv)\big)u\Big]\Big\|_{L_{t,x}^\frac{2(d+1)}{d+3}(I\times\R^d)}+
\Big\||\nabla|^\frac12\Big[\big(|x|^{-4}\ast|u|^2\big)v\Big]\Big\|_{L_{t,x}^\frac{2(d+1)}{d+3}(I\times\R^d)}\\
\lesssim&\|u\|_{X(I)}^2\|v\|_{X(I)},
\end{split}
\end{equation}
where $X(I)$ is defined to be
\begin{equation}\label{xidy}
X(I)=L_t^{d+1}L_x^\frac{2d(d+1)}{d^2-d-4}\bigcap
L_t^\frac{2(d+1)}{d-1}\dot{W}^{\frac12,\frac{2(d+1)}{d-1}}(I\times\R^d).
\end{equation}
\end{lemma}

\begin{proof}
By Lemma \ref{moser} and Hardy-Littlewood-Sobolev inequaltiy, we
obtain
\begin{align*}
&\Big\||\nabla|^\frac12\Big[\big(|x|^{-4}\ast(uv)\big)u\Big]\Big\|_{L_{t,x}^\frac{2(d+1)}{d+3}}\\
\leq&\big\||x|^{-4}\ast(uv)\big\|_{L_{t,x}^\frac{d+1}2}\big\||\nabla|^\frac12u\big\|_{L_{t,x}^\frac{2(d+1)}{d-1}}+
\big\||\nabla|^\frac12\big(|x|^{-4}\ast(uv)\big)\big\|_{L_t^2L_x^\frac{d}2}\|u\|_{L_t^{d+1}L_x^\frac{2d(d+1)}{d^2-d-4}}\\
\lesssim&\big\|uv\big\|_{L_t^\frac{d+1}2L_x^\frac{d(d+1)}{d^2-d-4}}\|u\|_{X(I)}+\big\||\nabla|^\frac12(uv)\big\|_{L_t^2L_x^\frac{d}{d-2}}\|u\|_{X(I)}\\
\lesssim&\|u\|_{X(I)}^2\|v\|_{X(I)}.
\end{align*}
Similarly, we can estimate another term.

\end{proof}

\subsection{Stability}
 Closely related to the continuous
dependence on the data, an essential tool for concentration
compactness
 arguments is the stability theory. More
precisely, given an approximate equation
\begin{equation}\label{near1}
\tilde{u}_{tt}-\Delta \tilde u=-f(\tilde{u})+e
\end{equation}
to \eqref{equ1}, with $e$ small in a suitable space and
$(\tilde{u}_0,\tilde{u}_1)$ is close to $(u_0,u_1)$ in energy space,
is it possible to show that solution $u$ to \eqref{equ1} stays very
close to $\tilde{u}$? Note that the question of continuous
dependence on the data corresponds to the case $e=0$.

The following  lemma for the nonlinear wave-Hartree equation is
analogs to the nonlinear Schr\"{o}dinger equation in \cite{CKSTT07}.
For convenience, we state the lemma and sketch its proof.

\begin{lemma}[Stability]\label{long}  Let $I$ be a
time interval, and let $\tilde{u}$ be a function on $I\times \R^d$
which is  a near-solution to \eqref{equ1}  in the sense that
\begin{equation}\label{near}
\tilde{u}_{tt}-\Delta \tilde u=-f(\tilde{u})+e
\end{equation}
for some function $e$. Assume that
\begin{equation}\label{eq2.20}\begin{aligned}
\|\tilde{u}\|_{S(I)}\leq M,\\ \|\tilde{u}\|_{L^\infty_t (I;\dot
H^1_x(\R^d))}+\|\partial_{t}\tilde{u}\|_{L^\infty_t (I; L^2_x(
\R^d))}\leq E
\end{aligned}
\end{equation}
for some constant $M,E>0$, where $S(I)$ is defined in
\eqref{scattersize}. Let $t_0\in I$, and let $(u(t_0),u_t(t_0))\in
\dot{H}^1\times L^2$ be close to $(\tilde{u}(t_0),\tilde{u_t}(t_0))$
in the sense that
\begin{equation}\label{eq2.21}
\|(u(t_0)-\tilde{u}(t_0),u_t(t_0)-\tilde{u}_t(t_0))\|_{\dot
H^1\times L^2}\leq\epsilon
\end{equation}
and assume also that the error term obeys
\begin{equation} \label{eq2.22}
\begin{split}
\big\||\nabla|^\frac12e\big\|_{L_{t,x}^\frac{2(d+1)}{d+3}(I\times\R^d)}\leq
\epsilon
\end{split}
\end{equation}
for some small $0<\epsilon<\epsilon_1=\epsilon_1( M, E)$. Then, we
conclude that there exists a solution $u:~I\times\R^d\to\R$ to
\eqref{equ1} with initial data $(u(t_0),u_t(t_0))$ at $t=t_0$, and
furthermore
\begin{equation}\label{eq2.23}
\begin{aligned}
\|u-\tilde{u}\|_{S^1(I)}\leq & C(M,E),\\
\|u\|_{S^1(I)}\leq & C(M,E),\\
\|u-\tilde{u}\|_{S(I)}\leq & C(M,E)\epsilon^c,
\end{aligned}
\end{equation}
where $c$ is a positive constant that depends on $d,~M$ and $E$, and
$S^1(I)$ is defined in \eqref{suoyimdy}.
\end{lemma}

\begin{proof}
Since $\|\tilde{u}\|_{S(I)}\leq M$, we may subdivide $I$ into $C(M,
\epsilon_0)$ time intervals $I_j$ such that
$$\|\tilde{u}\|_{S(I_j)}\leq \epsilon_0\ll1, \quad \quad 1\le j\le C(M,\epsilon_0). $$
 By the Strichartz estimate and standard bootstrap argument, we
have
\begin{equation*}\label{}
\|\tilde{u}\|_{S^1(I_j)}\leq C(E), \qquad 1\le j\le
C(M,\varepsilon_0).
\end{equation*}
Summing up over all the intervals, we obtain that
\begin{equation}\label{eq2.24}
\|\tilde{u}\|_{S^1(I)}\leq C(E, M).
\end{equation}
In particular, we have
\begin{equation}\label{eq2.25}
\|\tilde{u}\|_{X(I)}\leq C(E, M),
\end{equation}
where $X(I)$ is defined in \eqref{xidy}. This implies that there
exists a partition of the right half of $I$ at $t_0$:
$$t_0<t_1<\cdots<t_N,~I_j=(t_j,t_{j+1}),~I\cap(t_0,\infty)=(t_0,t_N),$$
such that $N\leq C(L,\delta)$ and for any $j=0,1,\cdots,N-1,$ we
have
\begin{equation}\label{ome}
\|\tilde{u}\|_{X(I_j)}\leq\delta\ll1.
\end{equation}
The estimate on the left half of $I$ at $t_0$ is analogue, we omit
it.

Let
\begin{equation}
\gamma(t)=u(t)-\tilde{u}(t),
\end{equation}
and \begin{equation}{\gamma_j(t) \choose
\dot{\gamma}_j(t)}=V_0(t-t_j){\gamma(t_j) \choose
\gamma_t(t_j)},~0\leq j\leq N-1,
\end{equation} then $\gamma$ satisfies the following
difference equation
\begin{align*}
\begin{cases}
(\partial_{tt}-\Delta){\gamma}=&-(V\ast|\tilde{u}|^2)\gamma-2\big[V\ast(\gamma
\tilde{u})\big]\tilde{u}-2\big[V\ast(\gamma
\tilde{u})\big]\gamma\\&-(V\ast|\gamma|^2)\tilde{u}-(V\ast|\gamma|^2)\gamma-e\\
\qquad\qquad\quad\triangleq&eq(\gamma)\\
{\gamma}(t_j)={\gamma}_j(t_j),&\gamma_t(t_j)=\dot{\gamma}_j(t_j),
\end{cases}
\end{align*}
which implies that
\begin{align*}
{\gamma(t)\choose \gamma_t(t)}=&{\gamma_j(t)\choose\dot{\gamma}_j(t)}-\int_{t_j}^t V_0(t-s){0\choose eq(\gamma)(s)}ds,\\
{\gamma_{j+1}(t) \choose \dot{\gamma}_{j+1}(t)}=&{\gamma_j(t)
\choose \dot{\gamma}_j(t)}-\int_{t_j}^{t_{j+1}}V_0(t-s){0\choose
eq(\gamma)(s)}ds.
\end{align*}
It follows from Lemma \ref{strichartzest} and Lemma \ref{nonle} that
\begin{align}\nonumber
\|\gamma-\gamma_j\|_{X(I_j)}+\|\gamma_{j+1}-\gamma_j\|_{X(I)}\lesssim&\sum_{k=1}^3\|\gamma\|_{X(I_j)}^k\|\tilde{u}\|_{X(I_j)}^{3-k}
+\big\||\nabla|^\frac12e\big\|_{L_{t,x}^\frac{2(d+1)}{d+3}(I_j\times\R^d)}\\\label{equ3}
\lesssim&\sum_{k=1}^3\|\gamma\|_{X(I_j)}^k\|\tilde{u}\|_{X(I_j)}^{3-k}+\epsilon.
\end{align}
Therefore, assuming that
\begin{equation}\label{laodong}
\|\gamma\|_{X(I_j)}\leq\delta\ll1,~\forall~j=0,1,\cdots,N-1,
\end{equation}
then by \eqref{ome} and \eqref{equ3}, we have
\begin{equation}
\|\gamma\|_{X(I_j)}+\|\gamma_{j+1}\|_{X(t_{j+1},t_N)}\leq
C\|\gamma_j\|_{X(t_j,t_N)}+\epsilon,
\end{equation}
for some absolute constant $C>0$. By \eqref{eq2.22} and iteration on
$j$, we obtain
\begin{equation}
\|\gamma\|_{X(I)}\leq (2C)^N\epsilon\leq\frac{\delta}2,
\end{equation}
if we choose $\epsilon_1$ sufficiently small. Hence the assumption
\eqref{laodong} is justified by continuity in $t$ and induction on
$j$. Then repeating the estimate \eqref{equ3} once again, we can get
the other Strichartz estimates on $u$.
\end{proof}

\vskip 0.2in

Using the above lemma  as well as its proof, one easily derives the
following local theory for \eqref{equ1}.

\begin{theorem}[Local well-posedness]\label{local}
Assume that $d\geq5$. Then, given  $(u_0,u_1)\in \dot
H^1(\R^d)\times L^2(\R^d)$ and $t_0\in\R$, there exists a unique
maximal-lifespan solution $u: I\times\R^d\to\R$ to \eqref{equ1} with
initial data $\big(u(t_0),u_t(t_0)\big)=\big(u_0,u_1\big)$. This
solution also has the following properties:
\begin{enumerate}
\item $($Local existence$)$ $I$ is an open neighborhood of $t_0$.
\item $($Blowup criterion$)$ If\ $\sup (I)$ is finite, then $u$ blows up
forward in time. If $\inf (I)$ is finite, then $u$ blows up backward
in time.
\item $($Scattering$)$ If $\sup (I)=+\infty$ and $u$ does not blow up
forward in time, then $u$ scatters forward in time in the sense
\eqref{1.2}. Conversely, given $(v_+,\dot{v}_+)\in \dot
H^1(\R^d)\times L^2(\R^d)$ there is a unique solution to
\eqref{equ1} in a neighborhood of infinity so that \eqref{1.2}
holds.
\item There
exists $\delta=\delta(d,\|(u_0,u_1)\|_{\dot{H}^1\times L^2})$ such
that if
$$\big\|\dot{K}(t-t_0)u_0 +
K(t-t_0)u_1\big\|_{S(I)}<\delta,$$ then, there exists a unique
solution $u:~I\times\mathbb{R}^d\to\R$ to \eqref{equ1}, with
$(u,\dot{u})\in C(I;\dot{H}^1\times L^2),$ and
\begin{equation*}
\|u\|_{S(I)}\leq 2\delta,~ \|u\|_{S^1(I)}<+\infty.
\end{equation*}
\item $($Small data global existence$)$ If
$\big\|(u_0,u_1)\big\|_{\dot H^1\times L^2}$ is sufficiently small
$($depending on $d)$, then $u$ is a global solution which does not
blow up either forward or backward in time. Indeed, in this case
$$S_\R(u)\lesssim\big\|(u_0,u_1)\big\|_{\dot H^1\times L^2}.$$
\end{enumerate}
\end{theorem}

\subsection{The Extended Causality}
In this subsection, we show a kind of the finite speed of
propagation named as causality to control the spatial center
function $x(t)$, which extends the result in Menzala-Strauss
\cite{MeS82}. As stated in the introduction, for the wave-Hartree
equation, we can not show the monotone local energy estimate on the
light cone. And as a substitute, Menzala-Strauss \cite{MeS82} show a
kind of the finite speed of propagation named as causality
 for the case $V\in
L^{\frac{d}{3}}+L^\infty$.

In fact, the causality can be improved,  this relies on two
important observations: one point is that the linear operators
$K(t)$ and $\dot K(t)$ still enjoy the finite speed of propagation,
the other point is that the Hartree term acted by the cut-off
function can be viewed as the linear feedback of the cutoff solution
in the cut-off Duhamel formulae (see \eqref{cut}) due to the
short-time Strichartz-norm boundness of the solution. The former
allows the cutoff function to go cross the linear operators $K(t)$
and $\dot K(t)$ and act directly on data and nonlinearity, while the
latter suggests us to iterate the solution just as the Gronwall
inequality not as the bootstrap argument. Based on the above
discussions, we can extend the exponent range of the causality to
 energy critical case $V\in
L^{\frac{d}4-}+L^\infty.$
\begin{lemma}[Extended Causality]\label{extend}
Assume that the data $(u_0,u_1)\in
 \dot{H}^1\times L^2$ have the compact support, i.e
\begin{align}{\rm{supp}}~ u_0, ~{\rm{supp}}~ u_1\subset \{x\in\Bbb R^d~:~|x|\leq R\}
\end{align}
for some constant $R>0$ and  $\big(u(t),\dot{u}(t)\big)\in
C([0,T_+(u_0,u_1)),\dot{H}^1\times L^2)$ is the finite energy
solution of the equation $(\ref{equ1})$ with initial data
$(u_0,u_1)$. Then it holds that
\begin{align}u(t,x)=0, ~~a.e.~~x\in\{x\in\Bbb R^d~:~|x|> R+t\},~\forall~t\in [0,T_+(u_0,u_1)).\end{align}

\end{lemma}

\begin{proof}
Let
\begin{align*}
\chi_t(x)=\begin{cases}1,\quad |x|>R+t,\\
0,\quad |x|\leq R+t.
\end{cases}
\end{align*}
From Duhamel's formula \eqref{inte1} and the finite speed of
propagation for the linear operators $K(t)$ and $\dot K(t)$, one has
\begin{equation}\label{cut}
\begin{split}
\chi_t(x)u(t)&=\chi_t(x)\dot{K}(t)u_0 +\chi_t(x)
K(t)u_1-\chi_t(x)\int^{t}_{0}K(t-s)f(u(s))ds\\
&=-\chi_t(x)\int^{t}_{0}K(t-s)(|x|^{-4}*|u(s)|^2)\chi_s(x)u(s)ds.
\end{split}
\end{equation}
For compact $[0,T]\subset \subset[0,T_+(u_0,u_1))$, by the
Strichartz estimate, we have
\begin{align*}
\big\|\chi_t(x)u(t)\big\|_{L^3_t([0,T],L_x^{\frac{6d}{3d-8}})}&\lesssim
\big\|(|x|^{-4}*|u|^2)\chi_t(x)u\big\|_{L^1_t([0,T],L^2_x)}\\
&\lesssim
\|u\|_{L^3_t([0,T],L_x^{\frac{6d}{3d-8}})}^2\big\|\chi_t(x)u\big\|_{L^3_t([0,T],L_x^{\frac{6d}{3d-8}})}.
\end{align*}
For $[0,T]\subset \subset[0,T_+(u_0,u_1))$, we can divide the
interval $[0,T]$ into $[0,T]=\cup_{j=1}^J I_j,$
 such that
$$\|u\|_{L^3_t(I_j,L_x^{\frac{6d}{3d-8}})}\leq \eta,$$
where $\eta$ is a small positive constant. And so
$$\big\|\chi_t(x)u(t)\big\|_{L^3_t(I_j,L_x^{\frac{6d}{3d-8}})}\leq
C\eta\big\|\chi_t(x)u(t)\big\|_{L^3_t(I_j,L_x^{\frac{6d}{3d-8}})},$$
thus
$$\big\|\chi_t(x)u(t)\big\|_{L^3_t(I_j,L_x^{\frac{6d}{3d-8}})}=0.$$
Summing up over all the intervals, we obtain that
\begin{equation}\label{ca1}
\big\|\chi_t(x)u(t)\big\|_{L^3_t([0,T],L_x^{\frac{6d}{3d-8}})}=0.\end{equation}
On the other hand, from the Strichartz estimates, we get
$$\big\|\chi_t(x)u(t)\big\|_{L_t^\infty([0,T],\dot H^1)}\lesssim
\|u\|_{L^3_t([0,T],L_x^{\frac{6d}{3d-8}})}^2\big\|\chi_t(x)u\big\|_{L^3_t([0,T],L_x^{\frac{6d}{3d-8}})}=0.$$
This together with \eqref{ca1} yields that
$$\chi_t(x)u(t)\equiv 0,\ a.e.\ x\in\mathbb{R}^d.$$
Therefore
$$u(t)\equiv 0,\ a.e.\ x\in\{x:|x|>
R+|t|\}.$$
\end{proof}

\section{Morawetz-type Estimate}
In this section, our task is to establish a useful Morawetz estimate
 by choosing a suitable multiplier, which plays an important role in excluding the almost periodic
solutions to \eqref{equ1}.  Noting that the nonlinearity of
\eqref{equ1} is a convolution term, we need use the quantity
$|u_\theta|^2/r+|u|^2/r^3$ to estimate $|u|^{2^*}$ inspired by
Nakanishi \cite{Na99e}. Compared with the local nonlinearity, we
explore a certain symmetry in nonlinearity to get a positive
integral quantity (this helps us to weak the requirement that the
integrand is positive) to deal with the nonlocal nonlinearity. The
embedding theorem for polar coordinates given in \cite{Na99e} is a
bridge connecting the whole space $\R^d$ with spherical surface
$\mathbb{S}^{d-1}$ in the proof.

\begin{proposition}[Morawetz estimate]\label{Morawetz}
Let u be a solution to \eqref{equ1} on a spacetime slab $I\times
\mathbb{R}^d$. then we have
\begin{eqnarray}\label{morawetzes}
\int_I\int_{\R^d}\frac{~~|u|^{2^*}}{|x|} dxdt\leq C(E),
\end{eqnarray}
where $2^\ast=\frac{2d}{d-2}$ and  $E$ is the energy $E(u_0,u_1)$.
\end{proposition}
{\bf Proof.}  
Let $\psi=u_r+\frac{(d-1)u}{2|x|}$, then we have
\begin{eqnarray*}
&\text{Re}\big\{\big((\pa_{tt}-\Delta)u+ (V\ast|u|^2)u\big)\bar \psi\big\}\\
=&\text{Re} \partial_t(\dot u \bar \psi)+ \text{Re}\nabla\cdot
\big\{-(\nabla u\bar
\psi)+\theta\ell(u)+\frac12|u|^2\nabla(\frac{(d-1)}{2|x|})\big\}
\\&+\frac{|u_\theta|^2}r+\frac{(d-1)(d-3)|u|^2}{4r^3}-\frac12\theta\cdot\nabla(V\ast|u|^2)|u|^2,
\end{eqnarray*}
where
\begin{equation*}
\begin{cases}
\ell (u)=\frac{1}{2}\big(-|\dot{u}|^2+|\nabla
u|^2+(V*|u|^2)|u|^2\big),\\
r=|x|,\quad \quad \quad \theta=\frac x {|x|},\\
u_r=\theta\cdot\nabla u,\quad u_{\theta}=\nabla u-\theta u_r.
\end{cases}
\end{equation*}

Integrating the above equality with respect to $(t, x)$ over
$$W=\Big\{(t,x)\Big|\;  0<a\leq t\leq b,\; x\in\R^d\Big\},$$
 we obtain that
\begin{equation}\label{eq3.1}
\int_W\Big[\frac{|u_\theta|^2}r+\frac{(d-1)(d-3)|u|^2}{4r^3}-\frac12|u|^2\frac
x{|x|}\cdot \nabla (V\ast |u|^2)\Big]dx dt \leq C
\bigg|\int_{\R^d}\dot u\bar \psi dx\big |_{t=a}^{t=b}\bigg|.
\end{equation}

Since
\begin{equation*}
 -\Big(\frac{x}{|x|}-\frac{y}{|y|}\Big)\nabla
V(x-y)=4\frac{|x||y|-x\cdot y}{|x-y|^6}\Big( \frac{1}{|x|}+
\frac{1}{|y|}\Big)\geq 0,
\end{equation*}
we have
\begin{equation*}
-\int_{a}^{b}\int_{\R^d\times\R^d}
\Big(\frac{x}{|x|}-\frac{y}{|y|}\Big)\nabla
V(x-y)|u(x)|^{2}|u(y)|^{2}dydxdt \geq 0.
\end{equation*}
This  implies that
\begin{equation}\label{eq3.2}
-\frac12\int_W|u|^2\frac x{|x|}\cdot \nabla (V\ast |u|^2)dx dt \geq
0.
\end{equation}
Substituting \eqref{eq3.2} into \eqref{eq3.1} and making use of the
Hardy inequality, we obtain that
\begin{equation}\label{eq3.3}
\begin{aligned}
\int_W\Big[\frac{|u_\theta|^2}r+\frac{(d-1)(d-3)|u|^2}{4r^3}\Big]dx
dt &\leq C \bigg|\int_{\R^d}\dot u\bar \psi dx\big
|_{t=a}^{t=b}\bigg|\\&\leq C\big(\|\dot{u}\|^2_2+\|\nabla
u\|_2^2\big)\\&\leq CE.
\end{aligned}
\end{equation}

On the other hand, we have
\begin{equation}\label{eq3.4}
\int_{\R^d}\frac{|u|^{2^*}}rdx=\int_0^\infty
r^{-1}\int_{\mathbb{S}^{d-1}}|u(r\theta)|^{2^*}d\theta r^{d-1}d
r=\int_0^\infty
r^{d-2}\|u(r\cdot)\|_{L^{2^*}_\theta(\mathbb{S}^{d-1})}^{2^*} d r.
\end{equation}
Note that the following interpolation and the Sobolev embedding
\begin{equation}\label{eq3.5}
\big[H^1(\mathbb{S}^{d-1}),
L^\beta(\mathbb{S}^{d-1})\big]_{\frac2d}=H^{\sigma}_q(\mathbb{S}^{d-1})\hookrightarrow
L^{2^*}(\mathbb{S}^{d-1}),
\end{equation}
where
\begin{equation*}
\begin{cases}
\beta=\frac{2(d-1)}{d-2},\quad \sigma=\frac{d-2}d,\\
\frac 1
q=\frac{\sigma}{d-1}+\frac1{2^*}=(1-\frac2d)\frac12+\frac2{d\beta},
\end{cases}
\end{equation*}
it follows that
\begin{equation}\label{eq3.6}
\int_{\R^d}\frac{|u|^{2^*}}rdx\leq\int_0^\infty
r^{d-2}\|u(r\cdot)\|_{H^1_\theta(\mathbb{S}^{d-1})}^{2}\|u(r\cdot)\|_{L^\beta_\theta(\mathbb{S}^{d-1})}^{\frac4{d-2}}
d r.
\end{equation}

Next we need  the following Sobolev embedding for the polar coordinates
given by Proposition 3.7  in \cite{Na99e}.
\begin{lemma}
Let $1\leq p <d.$ Then
\begin{equation}\label{eq3.7}
\dot W^{1,p}\hookrightarrow
L^\beta_{\theta}L^{\infty,\nu}_r\hookrightarrow L^{\infty,\nu}_r
L^\beta_{\theta},
\end{equation}
where $\beta=\frac{(d-1)p}{d-p}$, $\nu=\frac{d-p}p$ and
$L_r^{\infty,\nu}=\big\{u: \|u\|_{L_r^{\infty,\nu}}=\|r^\nu
u(r)\|_{L_r^\infty}<\infty$\big\}.
\end{lemma}
Applying this lemma with $p=2$ to \eqref{eq3.6}, we have
\begin{equation}\label{eq3.8}
\begin{aligned}
\int_{\R^d}\frac{|u|^{2^*}}r dx
\leq&\|r^{\frac{d-2}2}u(r\theta)\|^{\frac4{d-2}}_{L^{\infty}_r
L^\beta_{\theta}}\int_0^\infty
r^{-3}\|u(r\cdot)\|_{H^1_\theta(\mathbb{S}^{d-1})}^{2}
r^{d-1} d r\\
 \leq&\|\nabla u\|^{\frac4{d-2}}_{L^2}\int_{\R^d}
\Big(\frac{|u|^{2}}{r^3} +\frac{|u_\theta|^2}{r}\Big) d x.
\end{aligned}
\end{equation}
Integrating \eqref{eq3.8}  with respect to time $t$ and using
\eqref{eq3.3}, we deduce that
\begin{eqnarray*}
\int_a^b\int_{\R^d}\frac{|u|^{2^*}}r dxdt\leq C(E).
\end{eqnarray*}
This  concludes Proposition \ref{Morawetz}.

\section{The potential energy concentration}

In this section, we will prove the potential energy concentration
for the almost periodic  solutions. The analog for the energy
super-critical wave equation was originally shown by Killip-Visan
\cite{KV11}.

\begin{proposition}[potential energy concentration]\label{pec}
 Assume that $u(t,x)$
is  an almost periodic solution to \eqref{equ1} with maximal
lifespan $I$. Then there exists a constant  $C=C(u)$ such that
\begin{equation}\label{pec1}
\int_J\int_{|x-x(t)|\leq
C/N(t)}|u(t,x)|^\frac{2d}{d-2}dxdt\gtrsim_u|J|
\end{equation}
uniformly for all interavals  $J=[t_1,t_2]\subset I$ with $t_2\geq
t_1+N(t_1)^{-1}.$
\end{proposition}

Before we prove the above lemma, we recall a lemma in \cite{KV11}.
It says that although it is not possible to obtain the lower bounds
on the norm of $u(t)$ for a single time $t$ as the nonlinear wave
equation relies on two independent initial data, this phenomenon
must be rare as follows.

\begin{lemma}[\cite{KV11}]\label{ljtdp} Assume that $u(t,x):I\times\R^d\to\R$
is  an almost periodic solution to \eqref{equ1}. Then, for any
$A>0$, there exists $\eta=\eta(u,A)>0$ such that
\begin{equation}
\big|\{t\in[t_0,t_0+AN(t_0)^{-1}]\cap I:~\|u(t)\|_{L_x^\frac{2d}{d-2}}\geq\eta\}\big|\geq\eta
N(t_0)^{-1}.
\end{equation}
\end{lemma}

{\bf The proof of Proposition \ref{pec}:} In view of Lemma
\ref{localconstant}, it suffices to show \eqref{pec1} for intervals
of the form $[t_0,t_0+\delta N(t_0)^{-1}]$ for some small fixed
$\delta>0.$ We denote $J(t_0):=[t_0,t_0+\delta N(t_0)^{-1}]$.

Furthermore,  \eqref{pec1} can be reduced to prove that there exists
$C=C(u)$ such that
\begin{equation}\label{pecjz1}
\int_{J(t_0)}\int_{|x-x(t)|\leq
C/N(t)}|u|^\frac{2d}{d-2}dxdt\gtrsim_u N(t_0)^{-1}\simeq_u |J(t_0)|,
\end{equation}
since  $J(t_0)=[t_0,t_0+\delta N(t_0)^{-1}]$. On the other hand,
this follows from \eqref{apsp} and Lemma \ref{ljtdp}
\begin{equation}
\bigg|\Big\{t\in
J(t_0):~\int_{|x-x(t)|\leq\frac{C}{N(t)}}|u(t,x)|^\frac{2d}{d-2}dx\gtrsim_u
1\Big\}\bigg|\gtrsim N(t_0)^{-1}\sim_u |J(t_0)|.
\end{equation}
Thus, we conclude the proof of Proposition \ref{pec}.

\section{The soliton-like solution and frequency cascade solution}
In this section, we will use the Morawetz estimate and the potential
energy concentration established in the above section  to preclude
the soliton-like solution and low-to-high cascade solution in the
sense of Theorem \ref{three}. More precisely, we will prove
\begin{theorem}\label{sf}
There are no global solutions to \eqref{equ1} that are soliton-like
or low-to-high cascade in the sense of Theorem \ref{three}.
\end{theorem}
\begin{proof}
It follows from Proposition \ref{pec}  that there exists a constant
$C=C(u)$ such that
\begin{equation}\label{pec123}
\int_J\int_{|x-x(t)|\leq
C/N(t)}|u(t,x)|^\frac{2d}{d-2}dxdt\gtrsim_u|J|
\end{equation}
uniformly for all interavals  $J=[t_1,t_2]\subset \R$ with $t_2\geq
t_1+N(t_1)^{-1}.$

On the other hand, by  the extend causality in Lemma \ref{extend}
and \eqref{apsp}, we have the following control on the spatial
center function $x(t)$
\begin{equation}\label{xtgj}
|x(t_1)-x(t_2)|\leq
|t_1-t_2|+CN(t_1)^{-1}+CN(t_2)^{-1},~\forall~t_1,t_2\in\R.\end{equation}
The similar result for the energy-supercritical
 wave equation can be found in
\cite{KV11}. Translating space so that $x(0)=0$, we obtain by the
Morawetz estimate \eqref{morawetzes}, $N(t)\geq1$ and
\eqref{pec123}, for $T\geq1$,
\begin{align*}
C(E)&\geq\int_{0}^T\int_{\R^d}\frac{|u(t,x)|^{\frac{2d}{d-2}}}{|x|}dxdt\\
&\geq
\int_{0}^T\int_{|x-x(t)|\leq\frac{C}{N(t)}}\frac{|u(t,x)|^{\frac{2d}{d-2}}}{|x|}dxdt\\
&\gtrsim_u\sum_{k=0}^{[T-1]}\frac1{1+k}\int_k^{k+1}\int_{|x-x(t)|\leq\frac{C}{N(t)}}|u(t,x)|^{\frac{2d}{d-2}}dxdt\\
&\gtrsim_u\sum_{k=0}^{[T-1]}\frac1{1+k}\simeq\ln{(1+ T )}.
\end{align*}
Choosing $T$ sufficiently large depending on $u$, we derive a
contradiction. Thus, we exclude the global almost periodic solutions
in the sense of Theorem \ref{three}.

\end{proof}

\section{The finite time blowup}
In this section,  we preclude the finite time blowup solutions in
the sense of Theorem \ref{three}. We adopt the Morawetz estimate to
get a contradiction.

\begin{theorem}[Absence of finite-time blowup
 solutions]\label{blowup}  There are no finite-time blowup solutions to
 \eqref{equ1} in the sense of Theorem \ref{three}.
\end{theorem}
\begin{proof}
We argue by contradiction. Assume that there exists a solution $u:
I\times\R^d\to\R$ which is a finite time blowup in the sense of
Theorem \ref{three}. By the time-reversal and time-translation
symmetries, we may assume that the solution blows up as
$t\searrow0=\inf I.$

For $T\in\big(0,\sup I\big)$, using Morawetz estimate and splitting
the integral into a sum of integrals over intervals of local
constancy,   we derive that
\begin{align}\nonumber
C(E)&\geq\int^{\sup
I}_T\int_{\R^d}\frac{|u(t,x)|^{\frac{2d}{d-2}}}{|x|}dxdt\\\nonumber
&\geq\sum_{j:~I_j\subset(T,\sup I)}
\int_{I_j}\int_{\R^d}\frac{|u(t,x)|^{\frac{2d}{d-2}}}{|x|}dxdt\\\label{ceit}
&\geq\sum_{j:~I_j\subset(T,\sup I)} \int_{I_j}\int_{|x-x(t)|\leq
C/N(t)}\frac{|u(t,x)|^{\frac{2d}{d-2}}}{|x|}dxdt,
\end{align}
where $I_j=[t_j,t_{j+1}]\subset(T,\sup I)$ is an interval of local
constancy as in Lemma \ref{localconstant}. On the other hand, by the similar argument as (4.2) in \cite{KV11}, we have the
control for $x(t)$ by
\begin{equation}\label{xt123}
|x(t_1)-x(t_2)|\leq|t_1-t_2|+\frac{C_u}{N(t_1)}+\frac{C_u}{N(t_2)},\quad\forall~t_1,~t_2\in(0,\sup I).
\end{equation}
This together with Corollary \ref{finitetime} yields that
$$|x(t_1)-x(t_2)|\leq|t_1-t_2|+\frac{C_u}{N(t_1)}+\frac{C_u}{N(t_2)}\to0,\quad\text{as}\quad t_1,~t_2\to 0+$$
which means the limit $\lim\limits_{t\to0+}x(t)$ exists. Thus, the
space translation symmetry implies that we can assume
\begin{equation}\label{xt0}
\lim_{t\to0+}x(t)=0.
\end{equation}Combining this with \eqref{xt123}, we deduce  that on such an
 interval of local constancy $I_j$, we have
 $$|x(t)|+\frac{C}{N(t)}\lesssim_u t_j+\frac{C}{N(t_j)},~\forall~t\in I_j.$$
This together with \eqref{ceit}, Proposition \ref{pec} and Corollary
\ref{finitetime}: $N(t)\geq\frac{C}{t}$ implies that
\begin{align*}
C(E)&\geq\sum_{j:~I_j\subset(T,\sup I)} \int_{I_j}\int_{|x-x(t)|\leq
C/N(t)}\frac{|u(t,x)|^{\frac{2d}{d-2}}}{|x|}dxdt\\
\gtrsim_u&\sum_j\frac1{t_j+\frac{C}{N(t_j)}}\int_{I_j}\int_{|x-x(t)|\leq
C/N(t)}|u(t,x)|^{\frac{2d}{d-2}}dxdt\\
\gtrsim_u&\sum_j\frac1{t_j}|I_j|\\
\gtrsim_u&\int^{\sup I}_T\frac{dt}{t}.
\end{align*}
And so  we derive a contradiction by taking $T$ close to $0$
depending on $u$. Thus, we exclude the finite time blowup solutions
in the sense of Theorem \ref{three}.

\end{proof}

\textbf{Acknowledgements}  The authors would like to thank Professor Terence Tao and anonymous referee for their
comments  which helped improve the paper greatly. The authors were supported by
the NSF of China under grant No.11171033, 11231006.

\begin{center}

\end{center}

\end{document}